\documentclass[11pt]{amsart}

\usepackage{amsmath, amssymb, mathtools, palatino, euler, epic, eepic,floatflt, microtype}
\usepackage{graphicx}
\usepackage[usenames]{xcolor}
\usepackage{float}

\usepackage[margin=1 in]{geometry}
\definecolor{darkblue}{rgb}{0,0,0.7}
\definecolor{darkred}{rgb}{0.7,0,0}
\usepackage[colorlinks,linkcolor=darkred, citecolor=darkblue,
pagebackref=true,pdftex]{hyperref}

\newcommand\defi[1]{\textit{\color{blue}#1}}    % Definitions in the text
\newtheorem{thm}{Theorem}[section]
\newtheorem{prop}[thm]{Proposition}
\newtheorem{lemma}[thm]{Lemma}

\newtheorem{defn}[thm]{Definition}

\begin{document}
\title{Extendable shellability for $d$-dimensional complexes on $d+3$ vertices}

\author{Jared Culbertson, Anton Dochtermann, Dan P. Guralnik, Peter F. Stiller}
\address{Sensors Directorate, Air Force Research Laboratory, Dayton, OH}
\email{jared.culbertson@us.af.mil}
\address{Department of Mathematics, Texas State University, San Marcos, TX}
\email{dochtermann@txstate.edu}
\address{Department of Electrical \& Systems Engineering, University of Pennsylvania, Philadelphia, PA}
\email{guraldan@seas.upenn.edu}
\address{Department of Mathematics, Texas A\&M University, College Station, TX}
\email{stiller@math.tamu.edu}


%\subjclass[2010]{}

\date{\today}

\begin{abstract}
We prove that for all $d \geq 1$ a shellable $d$-dimensional complex with at most $d+3$ vertices is extendably shellable. The proof involves considering the structure of `exposed' edges in chordal graphs as well as a connection to linear quotients of quadratic monomial ideals.  
\end{abstract}

\maketitle
\section{Introduction}

A pure $d$-dimensional simplicial complex $\Delta$ is said to be \defi{shellable} if there exists an ordering of the facets $F_1, F_2, \dots, F_s$ such that for all $k = 2,3, \dots, n$ the simplicial complex induced by
\[\big(\bigcup_{i=1}^{k-1} F_i  \big ) \cap F_ k\]
\noindent
is pure of dimension $d-1$.    Shellability is an important combinatorial tool that has consequences for the topology of $\Delta$ as well as algebraic properties of its Stanley-Reisner (face) ring ${\mathbb K}[\Delta]$. Examples of shellable simplicial complexes include the independence complexes of matroids \cite{ProBil}, boundary complexes of simplicial polytopes \cite{BruMan}, as well as the skeleta of shellable complexes \cite{BjoWac}.  In particular for any $k = 1, 2, \dots, n-1$ the $k$-skeleton of a simplex on vertex set $[n]$ is shellable. It was recently shown \cite{GPPTW} that for every $d \geq 2$ deciding if a given pure $d$-dimensional simplicial complex is shellable is NP-hard.

Given a shellable complex a natural question to ask is whether one can get `stuck' in the process of building a shelling order.  A shellable complex $\Delta$ is said to be \defi{extendably shellable} if any shelling of a subcomplex of $\Delta$ can be extended to a shelling of $\Delta$.  Here a subcomplex of $\Delta$ is a simplicial complex $\Gamma$ whose set of facets consists of a subset of the facets of $\Delta$.  Any $2$-dimensional triangulated sphere (which is necessarily polytopal) is extendably shellable \cite{DanKle}, and Kleinschmidt \cite{Kle} has shown that any $d$-dimensional sphere with $d+3$ vertices is extendably shellable. Bj\"orner and Eriksson \cite{BjoEri} proved that independence complexes of rank 3 matroids are extendably shellable. On the other hand Ziegler \cite{Zie} has shown that there exist simplicial 4-polytopes that are not extendably shellable.

Simon \cite{Sim} has conjectured that every $k$-skeleton of a simplex is extendably shellable. One can see that the case of $k=0,1$ as well as $k=n-1, n-2$ are easy exercises. The Bj\"orner-Eriksson result establishes the $k=2$ case of Simon's conjecture by considering $U^3_n$, the uniform matroid of rank 3.  In \cite{BPZDec} Simon's conjecture was established for $k = n-3$ (a simpler proof was provided independently in \cite{Doc} based on results from \cite{CulGurSti}). Here we prove that much more is true.

\begin{thm} \label{thm:main}
Suppose $X$ is a shellable $d$-dimensional simplicial complex on at most 
$d+3$ vertices.  Then $X$ is extendably shellable.
\end{thm}
This of course implies the $k = n-3$ case of Simon's conjecture and also provides a generalization of Kleinschmidt's results.  Our result is also best possible in the sense that there are 2-dimensional complexes on 6 vertices that are not extendably shellable (see \cite{MorTak}, \cite{Bjo}).  

Again the statement is easy to establish if $X$ has $d+1$ or $d+2$ vertices so our result concerns the case of $d+3$ vertices. To prove Theorem \ref{thm:main} we use the correspondence between shellings of $d$-dimensional simplicial complexes on $d+3$ vertices and linear quotients of monomial ideals generated by quadrics.  In previous work of the authors (\cite{CulGurSti}, \cite{Doc}) it is shown that such constructions are equivalent to removing exposed edges from chordal graphs. Using these ideas we will see that Theorem \ref{thm:main} follows from the following graph-theoretic result.

\begin{prop} \label{prop:main}
Suppose $G$ is a chordal graph and suppose $H \subset G$ is a subgraph of $G$ that is also chordal.  Then $H$ can be obtained from $G$ via a sequence of removals of exposed edges.
\end{prop}

In the next section we recall the relevant definitions and set notation. In Section \ref{Sec:proofs} we prove Proposition \ref{prop:main} and show how it leads to a proof of Theorem \ref{thm:main}.

{\bf Acknowledgements.} We thank an anonymous referee who provided helpful comments on a earlier draft of the paper.

\section{Notation}
 For a finite set $E$ we let $\binom{E}{k}$ denote the family of $k$-subsets of $E$.  A subset $C \subset \binom{E}{k}$ will be called a \defi{$k$-clutter} on vertex set $E$.  For a $k$-clutter $C$ we use $\overline C$ to denote the (pure) $(k-1)$-dimensional simplicial complex generated by $C$, that is the collection of all subsets of elements of $C$ (including the empty set $\emptyset$).  Given a $k$-clutter $C$ a \defi{shelling step} is the addition of a some $e \in \binom{E}{k} - C$ to $C$ such that
\[\overline e - \overline C = \{f: d \subset f \subset e\},\]
\noindent
for some $d \subset e$.  In other words the intersection of $\overline e$ with $\overline C$ is a pure simplicial complex of dimension $k-2$.  

%Given clutters $C_1 \subset C_2$ we write $C_1 \nearrow C_2$ if there exists a sequence of shelling steps extending $C_1$ to $C_2$.  With these notations note that a clutter $X$ (or more precisely the simplicial complex $\overline X$) is \emph{shellable} if $\emptyset \nearrow X$ and is \emph{extendably shellable} if also for all $A \subset X$, we have that $\emptyset \nearrow A$ implies $A \nearrow X$.

If $\Delta$ is pure $d$-dimensional simplicial complex $\Delta$ with shelling order of its facets $(F_1, F_2, \dots, F_s)$, the \defi{restricted set} of the facet $F_i$ is the set of $(d-1)$-dimensional faces in the intersection of the facet $F_i$ with the subcomplex $F_1 \cup F_2 \cup \cdots \cup F_{i-1}$.

We next recall some basic notions from graph theory. For us a \defi{graph} $G$ consists of a finite set of vertices $V(G)$  along with a set $E(G)$ of unordered pairs of elements of $V(G)$. In particular our graphs are undirected and simple, with no loops or multiple edges. An element of $E(G)$ will be written $vw$, with set brackets and comma suppressed.  If $G$ is a graph on vertex set $V(G)$ then a \defi{subgraph} $H \subset G$ is a graph on a vertex set $V(H) \subset V(G)$ with the property that if $e \in E(H)$ then $e \in E(G)$. An \defi{induced} subgraph $H \subset G$ has the property that whenever $v,w \in V(H)$ and $vw \in E(G)$ we have $vw \in E(H)$.  If $G$ is a graph its \defi{complement graph} $G^C$ has the same vertex set $V(G)$ and edge set consisting of the edges not present in $G$.

If $S \subset V(G)$ is a subset of vertices of $G$ we let $G[S]$ denote the subgraph induced by $S$.  A \defi{clique} in a graph $G$ is an induced subgraph $K \subset G$ with the property that all pairs of vertices in $V(K)$ form edges in $K$ (so that $K$ is a \emph{complete graph}). If $v \in V(G)$ is a vertex of $G$ the \defi{neighborhood} of $v$ in $G$ is defined as
\[N_G(v) = \{w \in V(G): vw \in E(G) \}.\]
A vertex $v$ is \defi{simplicial} if the subgraph induced on $N(V)$ is a complete graph. We recall the notion of \emph{exposed edges} introduced in \cite{CulGurSti}.  

\begin{defn}
Suppose $G$ is a graph.
An edge $e \in E(G)$ is said to be \defi{exposed} if it is uniquely contained in a maximal clique.
If in addition the edge is properly contained in the clique (i.e. $e$ is contained in some triangle) then we say that $e$ is \defi{properly exposed}.
We refer to the operation of removing a (properly) exposed edge from $G$ as an \defi{(proper) erasure}.
A sequence of edges $(e_1,\ldots,e_k)$ is an \defi{(proper) erasure sequence}, if $e_i$ is a (properly) exposed edge of the graph $G-\{e_s\,|\,s<i\}$ for all $i=1,\ldots,k$.
\end{defn}

%We let $\partial G$ denote the set of exposed edges of $G$.
%
An edge which is exposed but not properly exposed will also be called a \defi{facet edge}.
A \defi{chordal graph} is a graph with no induced cycles of length four or more.
In \cite{CulGurSti} it is shown that a graph $G$ on vertex set $[n] = \{1,2, \dots, n\}$ is chordal if and only if $G$ can be obtained from the complete graph $K_n$ via a sequence of erasures.
Furthermore, this $G$ is connected (and chordal) if and only if each edge in that sequence is \emph{properly} exposed. 

We note that if we fix $V(G) = [n]=  \{1,2,\dots, n\}$ to be the vertex set of the graph $G$, the collection of edges $E(G) = \{e_1, e_2, \dots e_m\}$ corresponds to a $(n-2)$-clutter $X(G) \subset \binom{[n]}{n-2}$ with $(n-2)$-subsets given by $F_i = [n] \backslash e_i$.  
This correspondence can of course be reversed, so that an $(n-2)$-clutter $X$ corresponds to a graph $G(X)$.
We will use the following lemma from \cite{Doc}.

\begin{lemma}\label{lem:shelling}
An ordered set of edges $e_1, e_2, \dots, e_k$ is an erasure sequence in $K_n$ (resulting in a chordal graph $G$) if and only if the corresponding ordered set $F_1, F_2, \dots, F_k$ is a sequence of shelling steps (resulting in the $(n-3)$-dimensional simplicial complex $\overline{X(G^C)}$).
An edge $e_i$ is \emph{properly} exposed if and only if the restricted set of $F_i$ consists of less than $n-2$ elements.
\end{lemma}

\section{Proofs}\label{Sec:proofs}

In this section we provide the proofs of Proposition \ref{prop:main} and Theorem \ref{thm:main}, which  will follow from a number of elementary graph-theoretic lemmas.
Having obtained Proposition \ref{prop:main}, we discovered that a slightly weaker result had already been proved as Lemma 2 of \cite{RosTarLue}, though that result does not characterize the class of edges allowed for removal.
The notion of exposed edges allows for a new and simplified treatment of chordality that is independent of any particular structural descriptions of a chordal graph (such as the one obtained from a fixed vertex elimination ordering), resulting in a much simpler proof of Proposition~\ref{prop:main} than that found in~\cite{RosTarLue}.
Believing this to be of some independent interest, we include a self-contained discussion here, in the form of Lemmas~\ref{lem:A}--\ref{lem:D}.

We first collect some simple observations regarding exposed edges which follow straight from the definitions:
\begin{lemma} \label{lem:A}
For any graph $G$ and any edge $xy \in E(G)$ the following are equivalent:
\begin{enumerate}
 \item
 $xy \in G$ is exposed (respectively, properly exposed).
 \item
 $N_G(x) \cap N_G(y)$ is complete (resp. complete and nonempty).
 \item
 $y$ is a simplicial (resp. simplicial and non-isolated) vertex of $N_G(x)$.
 \item
 $x$ is a simplicial (resp. simplicial and non-isolated) vertex of $N_G(y)$.
\end{enumerate}
\end{lemma}
% 
%\begin{proof}
%The equivalence $(1) \Leftrightarrow (2)$ follows from Lemma 7 of \cite{CulGurSti}. Here we remark that for any two vertices $v, w \in G$ the intersection $N_G(v) \cap N_G(w)$ is just the union of all maximal cliques which contain both $v$ and $w$, minus $\{v, w\}$. The result then follows immediately from this observation and the definitions. 
%
%The fact that $(3)$ and $(4)$ are equivalent to $(2)$ easily from the definitions.   
%\end{proof}
The basic relationship between exposed edges and chordality is as follows:
\begin{lemma}\label{lem:B}
Suppose $G$ is a chordal graph.
Then:
\begin{enumerate}
 \item An edge $e \in E(G)$ is exposed if and only if $G - e$ is chordal;
 \item if $v \in V (G)$ is simplicial, every edge $xv \in E(G)$ is exposed.
\end{enumerate}
\end{lemma}

\begin{proof}
For the implication $\Rightarrow$ of $(1)$ we follow the proof of Theorem 8 in \cite{CulGurSti}.
Suppose $G$ is a chordal graph and $xy \in \partial G$ is an exposed edge.
If $C$ is an induced cycle in $H:= G - e$ such that $\{x,y\} \not\subset C$ (i.e., $C$ possibly contains $x$ or $y$, but not both), then $C$ is also an induced cycle of $G$ and so of length 3.
Otherwise, suppose $\{x,y\} \subset C$ and $|C| > 3$. Note that if $|C| > 4$, then the induced subgraph $C^\prime = C \cup xy$ of $G$ has an induced cycle of length greater than $3$, a contradiction.
This leaves us with the case where $C = x-v_1-y-v_2-x$ for some $v_1, v_2$. Since $xy$ is exposed in $G$, we must have $v_1v_2 \in E(G)$, since otherwise $xy$ would lie in two distinct maximal cliques so that $xy$ would not be exposed.
However, $v_1v_2 \in E(G)$ (hence in $E(H)$) means that $C$ would not be an induced cycle in $H$, a contradiction.

Conversely, note that if $e = xy$ is not exposed, then there are two non adjacent vertices $z, w \in N_G(x) \cap N_G(y)$.
Thus $x-z-y-w-x$ is an induced four cycle in $G-e$.

For $(2)$, notice that if $v \in G$ is simplicial then every vertex $x \in N_G(v)$ is simplicial in $N_G(v)$.
Thus either $N_G(v) = \{x\}$ so that $xv$ is a facet edge, or else Lemma \ref{lem:A} gives the result.
\end{proof}

The following result for chordal graphs is well-known, but is usually derived in the literature from the characterization of chordality via vertex elimination orderings.
Here we derive it directly from the hereditary property of chordality and the preceding characterization of exposed edges.
\begin{lemma}\label{lem:C}
Let $G$ be a chordal graph. Then:
\begin{enumerate}
    \item every facet edge in $G$ is a cut edge;
    \item either $G$ is complete or $G$ has at least two nonadjacent simplicial vertices.
\end{enumerate}
\end{lemma}

\begin{proof} To prove (1), consider a facet edge $wz$ in $G$.
If $H:=G-wz$ is connected, then there is a path from $z$ to $w$ in $G-wz$.
Since $G$ is chordal, we conclude there is a vertex $u\in N_G(w)\cap N_G(z)$.
But this contradicts $wz$ being a facet edge.

To prove (2), suppose $G$ is a counterexample with $f(G) := |E(G)| + |V(G)|$ as small as possible.
Then $G$ is chordal, incomplete (obviously), and connected:
otherwise, each connected component is a chordal graph that is not a counterexample and we can take one simplicial vertex from each component to obtain an independent set of simplicial vertices of cardinality at least 2.

Also, $G$ may not contain a facet edge $wz$ since otherwise $G-wz$ would be chordal with $f(G-wz) < f(G)$, contradicting minimality:
indeed, since $G-wz$ is disconnected by (1), each connected component $K$ of $G-wz$ would contain at least one simplicial vertex $u_K$ that is neither $w$ nor $z$, and $u_K$ is necessarily simplicial in $G$.

Next, we claim that $G$ has a properly exposed edge. To see this let $v \in V(G)$ be any vertex of $G$. Thus $f(N_G(v)) < f(G)$ and $G[N_G(v)]$ is chordal (all induced subgraphs of a chordal graph are chordal). Therefore either (i) $G[N_G(v)]$ is complete, meaning $v$ is a simplicial vertex and every edge incident to $v$ is exposed (and there must be one such edge because $G$ is connected and not complete); or (ii) $N_G(v)$ contains two simplicial vertices, each of which gives rise to a properly exposed edge in $G$ that is incident to $v$.
 
Now, let $e = xy$ be a properly exposed edge in $G$. Then $G-e$ is not complete and is chordal, and $f(G - e) < f(G)$.  Hence there exist $v_1, v_2$ nonadjacent simplicial vertices in $G - e$. Since $N_{G-e}(v_i)$ is complete, we have that $\{x, y\} \in N_{G-e}(v_i)$ and so $G[N_G(v_i)] = G[N_{G-e}(v_i)]$. Thus $v_1, v_2$ are also simplicial in $G$.
\end{proof}

The following lemma is the promised (mild) strengthening of Lemma 2 in \cite{RosTarLue}, with a new and simplified proof using exposed edges rather than perfect elimination orderings.
\begin{lemma} \label{lem:D}
Suppose $H$ and $G$ are chordal graphs with $V(H) = V(G)$, $E(H) = E$, $E(G) = E \cup F$ where $E \cap F = \emptyset$. Then $F$ contains an exposed edge of $G$. Equivalently, there exists an edge $e$ of $G$ such that $G - e$ is chordal and contains $H$.
\end{lemma}

\begin{proof}
Suppose that $(G, H)$ is a counterexample with $|V (G)|$ minimal. Then Lemma \ref{lem:C} ensures the existence of a simplicial vertex $z$ in $H$ (i.e., $N_H(z)$ is complete). First, suppose that $N_H(z) = N_G(z)$, so that $z$ is simplicial in $G$ as well and $G-z$, $H-z$ are chordal (Lemma \ref{lem:B}). Since $(G, H)$ is minimal, we can find an exposed edge $e = xy$ in $G-z$ such that $e \notin E(H-z)$. To see that $e$ is also exposed in $G$, notice that since $e \notin E(H-z)$ and hence $e \notin E(H)$ and $z$ is simplicial in $H$, we must have that $\{x, y\} \not\subset  N_H(z) = N_G(z)$, or in other words $z \notin N_G(x) \cap N_G(y)$. Thus G$[N_G(x) \cap N_G(y)] = G[N_{G-z}(x) \cap N_{G-z}(y)]$, which is complete and $e$ is exposed in $G$.

Now if $N_H(z) \not\subset N_G(z)$, then because $G[N_G(z)]$ is chordal we can apply Lemma \ref{lem:C} to see that either (i) $G[N_G(z)]$ is complete, in which case $z$ is simplicial in $G$ and any edge $zx$ with $x \in N_G(z) \backslash N_H(z)$ is exposed (Lemma \ref{lem:B}); or (ii) $G[N_G(z)]$ has two non-adjacent simplicial vertices $v_1,v_2$. Since $z$ is simplicial in $H$, we must have $\{v_1, v_2\} \not\subset N_H(z)$ and so one of $xv_1, xv_2 \notin E(H)$, but both are exposed in G.
\end{proof}

We can now prove the results stated in the introduction.

\begin{proof}[Proof of Proposition \ref{prop:main}]
Suppose $G$ is a chordal graph and suppose $H \subset G$ is a subgraph of $G$ that is also chordal.  By Lemma \ref{lem:D} there exists and edge $e$ of $G$ such that $G - e$ is chordal and contains $H$.  Continue removing edges in this way until we obtain the graph $H$.
\end{proof}

\begin{proof}[Proof of Theorem \ref{thm:main}]
Suppose $X$ is a shellable $d$-dimensional complex on vertex set $V = V(X)$.  We first note that if $|V| = d+1$ then $X$ is a simplex, and if $|V| = d+2$  it is not hard to see that any ordering on the facets of $X$ is a shelling: after we add the first facet $F_1$, every subsequent facet $F_i$ intersects the previous collection in $i-1$ faces of dimension $d-1$.  

Hence we can assume that $|V| = d+3$.  Let $(F_1, F_2, \dots, F_\ell)$ be some shelling order for $X$, where each $F_i$ is a facet of $X$ (a subset of $V$ of cardinality $d+1$).  Let $H = G(X)$ be the corresponding graph on the same vertex set, with edges given by $e_i = V \backslash F_i$. By Lemma \ref{lem:shelling} the graph $H$ is chordal.  Now suppose that we have a partial shelling $(F_{i_1}, F_{i_2}, \dots F_{i_k})$ of $X$ resulting in a subcomplex $Y \subset X$.  Again by Lemma \ref{lem:shelling} we have that $e_{i_1}, e_{i_2}, \dots, e_{i_k}$ is an erasure sequence in the complete graph $K_{d+3}$, resulting in the graph $G = G(Y)$ which by Lemma \ref{lem:B} is chordal.  Hence we see that $H$ is a chordal subgraph of the chordal graph $G$, both on the same vertex set $V$. By Proposition $\ref{prop:main}$ we can obtain $H$ by removing exposed edges in $G$, which (again by Lemma \ref{lem:shelling}) corresponds to completing the partial shelling into a shelling of $X$.  The result follows.
\end{proof}

\section{Conclusion}
Having established (a strengthening of) Simon's conjecture for $d$-dimensional complexes on $d+3$ vertices a natural question to ask is how the methods might extend to say $d+4$ vertices. Our work here relies on the the fact that a chordal graph $G$ always admits a exposed edge, and in fact we can remove exposed edges to obtain any chordal subgraph $H \subset G$.  Removing exposed edges correspond to making shelling moves on the `complementary' simplicial complex.

One can consider higher-dimensional analogues of these concepts in the context of $d$-clutters and of \emph{exposed circuits}.  In fact Simon's conjecture is equivalent to the statement that every $d$-clutter obtained by removing exposed circuits from the complete clutter $K_n^d$ admits an exposed circuit \cite{Doc}.  In the case of $d=2$ we can rely on various other properties of chordal graphs including the existence of simplicial vertices. Unfortunately these properties don't extend to higher dimensional.  In particular clutters obtained from removing exposed circuits from $K_n^d$ do not admit a simplicial vertex, or even a simplicial ridge \cite{BenBol}.

In \cite{BPZDec} the authors consider a notion of a \emph{decomposable} $d$-clutter that also generalize chordal graphs in a different way, and conjecture that such clutters admit a simplicial ridge.  As spelled out in \cite{BPZDec} this would imply Simon's conjecture.  As far as we know this conjecture is open even for the case of $d=3$.

\end{document}